\documentclass[preprint,12pt]{elsarticle}

\usepackage{amsmath}
\usepackage{amssymb}
\usepackage{amsthm}

\newtheorem{thm}{Theorem}[section]

\newtheorem{lem}[thm]{Lemma}

\newtheorem{defin}[thm]{Definition}
\newtheorem{rem}[thm]{Remark}
\numberwithin{equation}{section}

\journal{}

\begin{document}

\begin{frontmatter}



\title{Morrey spaces for Schr\"odinger operators with certain nonnegative potentials, Littlewood-Paley and Lusin functions on the Heisenberg groups}

%% \address[label1]{}[Wang,Hua]

\author{Hua Wang}
\address{School of Mathematics and Systems Science, Xinjiang University,\\
Urumqi 830046, P. R. China}

\begin{abstract}
Let $\mathcal L=-\Delta_{\mathbb H^n}+V$ be a Schr\"odinger operator on the Heisenberg group $\mathbb H^n$, where $\Delta_{\mathbb H^n}$ is the sublaplacian on $\mathbb H^n$ and the nonnegative potential $V$ belongs to the reverse H\"older class $RH_q$ with $q\geq Q/2$. Here $Q=2n+2$ is the homogeneous dimension of $\mathbb H^n$. Assume that $\{e^{-s\mathcal L}\}_{s>0}$ is the heat semigroup generated by $\mathcal L$. The Littlewood-Paley function $\mathfrak{g}_{\mathcal L}$ and the Lusin area integral $\mathcal{S}_{\mathcal L}$ associated with the Schr\"odinger operator $\mathcal L$ are defined, respectively, by
\begin{equation*}
\mathfrak{g}_{\mathcal L}(f)(u):=\bigg(\int_0^{\infty}\bigg|s\frac{d}{ds}e^{-s\mathcal L}f(u)\bigg|^2\frac{ds}{s}\bigg)^{1/2}
\end{equation*}
and
\begin{equation*}
\mathcal{S}_{\mathcal L}(f)(u):=\bigg(\iint_{\Gamma(u)}\bigg|s\frac{d}{ds}e^{-s\mathcal L}f(v)\bigg|^2\frac{dvds}{s^{Q/2+1}}\bigg)^{1/2},
\end{equation*}
where
\begin{equation*}
\Gamma(u):=\big\{(v,s)\in\mathbb H^n\times(0,\infty):|u^{-1}v|<\sqrt{s\,}\big\}.
\end{equation*}
In this paper the author first introduces a class of Morrey spaces associated with the Schr\"odinger operator $\mathcal L$ on $\mathbb H^n$. Then by using some pointwise estimates of the kernels related to the nonnegative potential $V$, the author establishes the boundedness properties of these two operators $\mathfrak{g}_{\mathcal L}$ and $\mathcal{S}_{\mathcal L}$ acting on the Morrey spaces. It can be shown that the same conclusions also hold for the operators $\mathfrak{g}_{\sqrt{\mathcal L}}$ and $\mathcal{S}_{\sqrt{\mathcal L}}$ with respect to the Poisson semigroup $\{e^{-s\sqrt{\mathcal L}}\}_{s>0}$.
\end{abstract}


\end{frontmatter}

\section{Introduction}

\subsection{The Heisenberg group $\mathbb H^n$}
This paper deals with Morrey spaces for Schr\"odinger operators with certain nonnegative potentials and Littlewood-Paley functions on the Heisenberg groups. The Heisenberg group is the most well-known example from the realm of nilpotent Lie groups and plays an important role in several branches of mathematics, such as representation theory, partial differential equations, several complex variables and harmonic analysis. It is a remarkable fact that the Heisenberg group, an important example of the simply-connected nilpotent Lie group, naturally arises in two fundamental but different
settings in modern analysis. On the one hand, it can be identified with the group of translations of the Siegel upper half space in $\mathbb C^{n+1}$ and plays an important role in our understanding of several problems in the complex function theory of the unit ball. On the other hand, it can also be realized as the group of unitary operators generated by the position and momentum operators in the context of quantum mechanics.

We begin by recalling some notions from \cite{folland3,stein2}. We write $\mathbb N=\{1,2,3,\dots\}$ for the set of natural numbers. The sets of real and complex numbers are denoted by $\mathbb R$ and $\mathbb C$, respectively. Let $\mathbb H^n$ be a \emph{Heisenberg group} of dimension $2n+1$; that is, a two-step nilpotent Lie group with underlying manifold $\mathbb C^n\times\mathbb R$. The group structure (the multiplication law) is given by
\begin{equation*}
(z,t)\cdot(z',t'):=\big(z+z',t+t'+2\mathrm{Im}(z\cdot\overline{z'})\big),
\end{equation*}
where $z=(z_1,z_2,\dots,z_n)$, $z'=(z_1',z_2',\dots,z_n')\in\mathbb C^n$, and $z\cdot\overline{z'}:=\sum_{j=1}^nz_j\overline{z_j'}$.
Under this multiplication $\mathbb H^n$ becomes a nilpotent unimodular Lie group. It is easy to see that the inverse element of $u=(z,t)\in \mathbb H^n$ is $u^{-1}=(-z,-t)$, and the identity is the origin $0=(0,0)$. The corresponding Lie algebra $\mathfrak{h}^n$ of left-invariant vector fields on $\mathbb H^n$ is spanned by
\begin{equation*}
\begin{cases}
X_j:=\frac{\partial}{\partial x_j}+2y_j\frac{\partial}{\partial t},\quad j=1,2,\dots,n;&\\
Y_j:=\frac{\partial}{\partial y_j}-2x_j\frac{\partial}{\partial t},\quad j=1,2,\dots,n;&\\
T:=\frac{\partial}{\partial t}.&
\end{cases}
\end{equation*}
All non-trivial commutation relations are given by
\begin{equation*}
[X_j,Y_j]=-4T,\quad j=1,2,\dots,n.
\end{equation*}
Here $[\cdot,\cdot]$ is the usual Lie bracket. The \emph{sublaplacian} $\Delta_{\mathbb H^n}$ is defined by
\begin{equation*}
\Delta_{\mathbb H^n}:=\sum_{j=1}^n\big(X_j^2+Y_j^2\big).
\end{equation*}
The Heisenberg group has a natural dilation structure which is consistent with the Lie group structure mentioned above. For each positive number $a>0$, we define the \emph{dilation} on $\mathbb H^n$ by
\begin{equation*}
\delta_a(z,t):=(az,a^2t),\quad (z,t)\in \mathbb H^n.
\end{equation*}
Observe that $\delta_a$ ($a>0$) is an automorphism of the group $\mathbb H^n$. For given $u=(z,t)\in\mathbb H^n$, the \emph{homogeneous norm} of $u$ is given by
\begin{equation*}
|u|=|(z,t)|:=\big(|z|^4+t^2\big)^{1/4}.
\end{equation*}
Observe that $|(z,t)^{-1}|=|(z,t)|$ and
\begin{equation*}
\big|\delta_a(z,t)\big|=\big(|az|^4+(a^2t)^2\big)^{1/4}=a|(z,t)|,\quad a>0.
\end{equation*}
In addition, this norm $|\cdot|$ satisfies the triangle inequality and then leads to a left-invariant distance $d(u,v)=|u^{-1}\cdot v|$ for any $u=(z,t)$, $v=(z',t')\in\mathbb H^n$. If $r>0$ and $u\in\mathbb H^n$, let $B(u,r)=\{v\in\mathbb H^n:d(u,v)<r\}$ be the (open) ball with center $u$ and radius $r$.
The Haar measure on $\mathbb H^n$ coincides with the Lebesgue measure $dzdt$ on $\mathbb C^n\times\mathbb R=\mathbb R^{2n}\times\mathbb R$. The measure of any measurable set $E\subset\mathbb H^n$ is denoted by $|E|$. For $(u,r)\in\mathbb H^n\times(0,\infty)$, it can be proved that the volume of $B(u,r)$ is
\begin{equation}\label{sqrts}
|B(u,r)|=r^{Q}\cdot|B(0,1)|,
\end{equation}
where $Q:=2n+2$ is the \emph{homogeneous dimension} of $\mathbb H^n$ and $|B(0,1)|$ is the volume of the unit ball in $\mathbb H^n$. A direct calculation shows that
\begin{equation*}
|B(0,1)|=\frac{2\pi^{n+\frac{\,1\,}{2}}\Gamma(\frac{\,n\,}{2})}{(n+1)\Gamma(n)\Gamma(\frac{n+1}{2})}.
\end{equation*}
In the sequel, when $B=B(u,r)$ in $\mathbb H^n$ and $\lambda>0$, we shall use the notation $\lambda B$ to denote the ball with the same center $u$ and radius $\lambda r$.Clearly, we have
\begin{equation}\label{homonorm}
|B(u,\lambda r)|=\lambda^{Q}\cdot|B(u,r)|,\quad (u,r)\in\mathbb H^n\times(0,\infty),\;\lambda\in(0,\infty).
\end{equation}
For various aspects of harmonic analysis on the Heisenberg group, we refer the readers to \cite[Chapter XII]{stein2}, \cite{folland}, \cite{thangavelu} and the references therein.

\subsection{The Schr\"odinger operator $\mathcal L$}
A nonnegative locally $L^q$ integrable function $V$ on $\mathbb H^n$ is said to belong to the \emph{reverse H\"older class} $RH_q$ for $1<q<\infty$, if there exists a positive constant $C=C(q;V)$ such that the reverse H\"older inequality
\begin{equation*}
\left(\frac{1}{|B|}\int_B V(w)^q\,dw\right)^{1/q}\leq C\cdot\left(\frac{1}{|B|}\int_B V(w)\,dw\right)
\end{equation*}
holds for every ball $B$ in $\mathbb H^n$. In this article we will always assume that $0\not\equiv V\in RH_q$ for $q\geq Q/2$ and $Q=2n+2$. We now consider the \emph{Schr\"odinger operator} with the potential $V\in RH_q$ on $\mathbb H^n$ (see \cite{lin}):
\begin{equation*}
\mathcal L:=-\Delta_{\mathbb H^n}+V.
\end{equation*}
In recent years, the investigation of Schr\"{o}dinger operators on the Euclidean space $\mathbb R^n$ with nonnegative potentials which belong to the reverse H\"{o}lder class has
attracted a lot of attention; see, for example, \cite{dziu1,dziu2,dziu3,dziu,hofman,shen}. For the weighted cases, see \cite{bong1,bong2,bong3,bong4,bui,song,tang}.
As in \cite{lin}, for given $V\in RH_q$ with $q\geq Q/2$, we introduce the \emph{critical radius function} $\rho(u)=\rho(u;V)$ which is defined by
\begin{equation}\label{rho}
\rho(u):=\sup\bigg\{r\in(0,\infty):\frac{1}{r^{Q-2}}\int_{B(u,r)}V(w)\,dw\leq1\bigg\},\quad u\in\mathbb H^n,
\end{equation}
where $B(u,r)$ denotes the ball in $\mathbb H^n$ centered at $u$ and with radius $r$. It is well known that the auxiliary function $\rho(u)$ determined by $V\in RH_q$ satisfies
\begin{equation*}
0<\rho(u)<\infty
\end{equation*}
for any given $u\in\mathbb H^n$ (see \cite{lin,lu}). We need the following known result concerning the critical radius function \eqref{rho}.
\begin{lem}[\cite{lu}]\label{N0}
If $V\in RH_q$ with $q\geq Q/2$, then there exist constants $C_0\geq 1$ and $N_0>0$ such that, for any $u$ and $v$ in $\mathbb H^n$,
\begin{equation}\label{com}
\frac{\,1\,}{C_0}\left[1+\frac{|v^{-1}u|}{\rho(u)}\right]^{-N_0}\leq\frac{\rho(v)}{\rho(u)}\leq C_0\left[1+\frac{|v^{-1}u|}{\rho(u)}\right]^{\frac{N_0}{N_0+1}}.
\end{equation}
\end{lem}
Lemma \ref{N0} is due to Lu \cite{lu} (see also \cite[Lemma 4]{lin}). In the setting of $\mathbb R^n$, this result was first given by Shen in \cite[Lemma 1.4]{shen}. As a direct consequence of \eqref{com}, we can see that for each fixed $k\in\mathbb N$, the following estimate
\begin{equation}\label{com2}
\left[1+\frac{r}{\rho(u)}\right]^{-\frac{N_0}{N_0+1}}\left[1+\frac{2^kr}{\rho(u)}\right]\leq C_0\left[1+\frac{2^kr}{\rho(v)}\right]
\end{equation}
holds for any $v\in B(u,2^kr)$ with $u\in\mathbb H^n$ and $r\in(0,\infty)$, $C_0$ is the same as in \eqref{com}. Let $\mathcal L=-\Delta_{\mathbb H^n}+V$ be a Schr\"{o}dinger operator on the Heisenberg group $\mathbb H^n$, where $\Delta_{\mathbb H^n}$ is the sublaplacian and the nonnegative potential $V$ belongs to the reverse H\"{o}lder class $RH_q$ with $q\geq Q/2$, and $Q$ is the homogeneous dimension of $\mathbb H^n$. Since $V$ is nonnegative and belongs to $L^q_{\mathrm{loc}}(\mathbb H^n)$, $\mathcal L$ generates a $(C_0)$ contraction semigroup $\big\{\mathcal T^{\mathcal L}_s\big\}_{s>0}=\big\{e^{-s\mathcal L}\big\}_{s>0}$. Let $\mathcal P_s(u,v)$ denote the kernel of the semigroup $\big\{e^{-s\mathcal L}\big\}_{s>0}$.
\begin{equation*}
\mathcal T^{\mathcal L}_sf(u)=e^{-s\mathcal L}f(u)=\int_{\mathbb H^n}\mathcal P_s(u,v)f(v)\,dv,\quad f\in L^2(\mathbb H^n),~s>0.
\end{equation*}
By the \emph{Trotter product formula} (see \cite{gold}), one has
\begin{equation}\label{heat}
0\leq \mathcal P_s(u,v)\leq C\cdot s^{-Q/2}\exp\bigg(-\frac{|v^{-1}u|^2}{As}\bigg),\quad s>0.
\end{equation}
Moreover, by using the estimates of fundamental solution for the Schr\"{o}dinger operator on $\mathbb H^n$, this estimate \eqref{heat} can be improved when $V$ belongs to the reverse H\"older class $RH_q$ for some $q\geq Q/2$. The auxiliary function $\rho(u)$ arises naturally in this context.
\begin{lem}[\cite{lin}]\label{ker1}
Let $V\in RH_q$ with $q\geq Q/2$, and let $\rho(u)$ be the auxiliary function determined by $V$. For every positive integer $N\in\mathbb N$, there exists a positive constant $C_N>0$ such that, for any $u$ and $v$ in $\mathbb H^n$,
\begin{equation*}
0\leq\mathcal P_s(u,v)\leq C_N\cdot s^{-Q/2}\exp\bigg(-\frac{|v^{-1}u|^2}{As}\bigg)\bigg[1+\frac{\sqrt{s\,}}{\rho(u)}+\frac{\sqrt{s\,}}{\rho(v)}\bigg]^{-N},\quad s>0.
\end{equation*}
\end{lem}

\begin{rem}
$(i)$ This estimate of $\mathcal P_s(u,v)$ is much better than \eqref{heat}, which was given by Lin and Liu in \cite[Lemma 7]{lin}.

$(ii)$ For the Schr\"{o}dinger operators in a more general setting (such as nilpotent Lie group), see, for example, \cite{jiang,li}.
\end{rem}

\subsection{Littlewood-Paley function and Lusin area integral}
The Littlewood-Paley functions play an important role in classical harmonic analysis, for example in the study of non-tangential convergence of Fatou type and boundedness of Riesz transforms and multipliers (see \cite{stein}). Assume that $\big\{e^{-s\mathcal L}:s>0\big\}$ is the semigroup generated by $\mathcal L$. The Littlewood-Paley function associated with the Schr\"{o}dinger operator $\mathcal L$ on the Heisenberg group is defined by (see \cite{lin})
\begin{equation*}
\begin{split}
\mathfrak{g}_{\mathcal L}(f)(u)&:=\bigg(\int_0^{\infty}\bigg|s\frac{d}{ds}e^{-s\mathcal L}f(u)\bigg|^2\frac{ds}{s}\bigg)^{1/2}\\
&=\bigg(\int_0^{\infty}\big|(s\mathcal L)e^{-s\mathcal L}f(u)\big|^2\frac{ds}{s}\bigg)^{1/2}.
\end{split}
\end{equation*}
We also consider the Lusin area integral associated with the Schr\"{o}dinger operator $\mathcal L$ on $\mathbb H^n$, which is defined by (see also \cite{lin})
\begin{equation*}
\begin{split}
\mathcal{S}_{\mathcal L}(f)(u)&:=\bigg(\iint_{\Gamma(u)}\bigg|s\frac{d}{ds}e^{-s\mathcal L}f(v)\bigg|^2\frac{dvds}{s^{Q/2+1}}\bigg)^{1/2}\\
&=\bigg(\iint_{\Gamma(u)}\big|(s\mathcal L)e^{-s\mathcal L}f(v)\big|^2\frac{dvds}{s^{Q/2+1}}\bigg)^{1/2},
\end{split}
\end{equation*}
where
\begin{equation*}
\Gamma(u):=\big\{(v,s)\in\mathbb H^n\times(0,\infty):|u^{-1}v|<\sqrt{s\,}\big\}.
\end{equation*}
Recall that in the setting of $\mathbb R^n$, these two integral operators were investigated by many authors (see \cite{bong1}, \cite{dziu}, \cite{hofman} and \cite{pan}). In this article we shall be interested in the behavior of the Littlewood-Paley function $\mathfrak{g}_{\mathcal L}$ and the Lusin area integral $\mathcal{S}_{\mathcal L}$ related to Schr\"odinger operator on $\mathbb H^n$.

For $1\leq p<\infty$, the Lebesgue space $L^p(\mathbb H^n)$ is defined to be the set of all measurable functions $f$ on $\mathbb H^n$ such that
\begin{equation*}
\big\|f\big\|_{L^p(\mathbb H^n)}:=\bigg(\int_{\mathbb H^n}|f(u)|^p\,du\bigg)^{1/p}<\infty.
\end{equation*}
The weak Lebesgue space $WL^1(\mathbb H^n)$ consists of all measurable functions $f$ defined on $\mathbb H^n$ such that
\begin{equation*}
\big\|f\big\|_{WL^1(\mathbb H^n)}:=
\sup_{\lambda>0}\lambda\cdot\big|\big\{u\in\mathbb H^n:|f(u)|>\lambda\big\}\big|<\infty.
\end{equation*}

Recently, Lin and Liu \cite{lin} established strong-type and weak-type estimates of the operators $\mathfrak{g}_{\mathcal L}$ and $\mathcal{S}_{\mathcal L}$ on the Lebesgue spaces. Their main results can be formulated as follows.
\begin{thm}[\cite{lin}]\label{strongg}
Let $1\leq p<\infty$. Then the following statements are valid$:$
\begin{enumerate}
  \item if $p>1$, then the operator $\mathfrak{g}_{\mathcal L}$ is bounded from $L^p(\mathbb H^n)$ to $L^p(\mathbb H^n);$
  \item if $p=1$, then the operator $\mathfrak{g}_{\mathcal L}$ is bounded from $L^1(\mathbb H^n)$ to $WL^1(\mathbb H^n)$.
\end{enumerate}
\end{thm}
\begin{thm}[\cite{lin}]\label{strongs}
Let $1\leq p<\infty$. Then the following statements are valid$:$
\begin{enumerate}
  \item if $p>1$, then the operator $\mathcal{S}_{\mathcal L}$ is bounded from $L^p(\mathbb H^n)$ to $L^p(\mathbb H^n);$
  \item if $p=1$, then the operator $\mathcal{S}_{\mathcal L}$ is bounded from $L^1(\mathbb H^n)$ to $WL^1(\mathbb H^n)$.
\end{enumerate}
\end{thm}

\begin{rem}
$(i)$ It was also proved by Lin and Liu that these two operators are bounded on $\mathrm{BMO}_{\mathcal L}(\mathbb H^n)$, and bounded from $H^1_{\mathcal L}(\mathbb H^n)$ to $L^1(\mathbb H^n)$.

$(ii)$ In \cite{zhao}, Zhao introduced and studied the Littlewood-Paley and Lusin functions associated to the sublaplacian operator on (connected) nilpotent Lie groups. The $L^p$ $(1<p<\infty)$ boundedness of Littlewood-Paley and Lusin functions are obtained in this general setting.
\end{rem}

The organization of this paper is as follows. In Section \ref{def}, we will give the definitions of Morrey space and weak Morrey space associated with Schr\"odinger operator on $\mathbb H^n$.In Section \ref{sec3}, we establish the boundedness properties of the Littlewood-Paley function $\mathfrak{g}_{\mathcal L}$ in the context of Morrey spaces. Section \ref{sec4} is devoted to proving the boundedness of the Lusin area integral $\mathcal{S}_{\mathcal L}$. The generalized Morrey estimates for the operators $\mathfrak{g}_{\mathcal L}$ and $\mathcal{S}_{\mathcal L}$ are obtained in Section \ref{sec5}.
All results hold for the operators $\mathfrak{g}_{\sqrt{\mathcal L}}$ and $\mathcal{S}_{\sqrt{\mathcal L}}$ with respect to the Poisson semigroup as well.

Throughout this paper, $C$ always denotes a positive constant which is independent of the main parameters involved but whose value may be different from line to line, and a subscript is added when we wish to make clear its dependence on the parameter in the subscript. The symbol $\mathbb{A}\lesssim \mathbb{B}$ means that $\mathbb{A}\leq C\mathbb{B}$ with some positive constant $C$. If $\mathbb{A}\lesssim \mathbb{B}$ and $\mathbb{B}\lesssim \mathbb{A}$, then we write $\mathbb{A}\approx \mathbb{B}$ to denote the equivalence of $\mathbb{A}$ and $\mathbb{B}$. For any $p\in[1,\infty)$, the notation $p'$ denotes its conjugate number, namely, $1/p+1/{p'}=1$ and $1'=\infty$.

\section{Definitions of Morrey and weak Morrey spaces}\label{def}
The classical Morrey space $M^{p,\lambda}(\mathbb R^n)$ was originally introduced by Morrey in \cite{morrey} to study the local behavior of solutions to second order elliptic partial differential equations. Since then, this space was systematically developed by many authors. Nowadays this space has been studied intensively and widely used in analysis, geometry, mathematical physics and other related fields. For the properties and applications of classical Morrey space, we refer the readers to \cite{adams1,adams2,adams3,fazio1,fazio2,taylor} and the references therein. We denote by $M^{p,\lambda}(\mathbb R^n)$ the Morrey space, which consists of all $p$-locally integrable functions $f$ on $\mathbb R^n$ such that
\begin{equation*}
\begin{split}
\|f\|_{M^{p,\lambda}(\mathbb R^n)}:=&\sup_{x\in\mathbb R^n,r>0}r^{-\lambda/p}\|f\|_{L^p(B(x,r))}\\
=&\sup_{x\in\mathbb R^n,r>0}r^{-\lambda/p}\bigg(\int_{B(x,r)}|f(y)|^p\,dy\bigg)^{1/p}<\infty,
\end{split}
\end{equation*}
where $1\leq p<\infty$ and $0\leq\lambda\leq n$. Note that $M^{p,0}(\mathbb R^n)=L^p(\mathbb R^n)$ and $M^{p,n}(\mathbb R^n)=L^\infty(\mathbb R^n)$ by the Lebesgue differentiation theorem. If $\lambda<0$ or $\lambda>n$, then $M^{p,\lambda}(\mathbb R^n)=\Theta$, where $\Theta$ is the set of all functions equivalent to 0 on $\mathbb R^n$. We also denote by $WM^{1,\lambda}(\mathbb R^n)$ the weak Morrey space, which consists of all measurable functions $f$ on $\mathbb R^n$ such that
\begin{equation*}
\begin{split}
\|f\|_{WM^{1,\lambda}(\mathbb R^n)}:=&\sup_{x\in\mathbb R^n,r>0}r^{-\lambda}\|f\|_{WL^1(B(x,r))}\\
=&\sup_{x\in\mathbb R^n,r>0}r^{-\lambda}\sup_{\sigma>0}\sigma\big|\big\{y\in B(x,r):|f(y)|>\sigma\big\}\big|<\infty.
\end{split}
\end{equation*}

In this section, we introduce some kinds of Morrey spaces associated with the Schr\"odinger operator $\mathcal L$ on $\mathbb H^n$.
\begin{defin}
Let $\rho$ be the auxiliary function determined by $V\in RH_q$ with $q\geq Q/2$. Let $1\leq p<\infty$ and $0\leq\kappa<1$. For given $0<\theta<\infty$, the Morrey space $L^{p,\kappa}_{\rho,\theta}(\mathbb H^n)$ is defined to be the set of all $p$-locally integrable functions $f$ on $\mathbb H^n$ such that
\begin{equation}\label{morrey1}
\bigg(\frac{1}{|B(u_0,r)|^{\kappa}}\int_{B(u_0,r)}|f(u)|^p\,du\bigg)^{1/p}
\leq C\cdot\left[1+\frac{r}{\rho(u_0)}\right]^{\theta}
\end{equation}
holds for every ball $B(u_0,r)$ in $\mathbb H^n$, $u_0$ and $r$ denote the center and radius of $B(u_0,r)$, respectively. A norm for $f\in L^{p,\kappa}_{\rho,\theta}(\mathbb H^n)$, denoted by $\|f\|_{L^{p,\kappa}_{\rho,\theta}(\mathbb H^n)}$, is given by the infimum of the constants appearing in \eqref{morrey1}, or equivalently,
\begin{equation*}
\big\|f\big\|_{L^{p,\kappa}_{\rho,\theta}(\mathbb H^n)}:=\sup_{B(u_0,r)}\left[1+\frac{r}{\rho(u_0)}\right]^{-\theta}
\bigg(\frac{1}{|B(u_0,r)|^{\kappa}}\int_{B(u_0,r)}|f(u)|^p\,du\bigg)^{1/p}
<\infty,
\end{equation*}
where the supremum is taken over all balls $B(u_0,r)$ in $\mathbb H^n$. Define
\begin{equation*}
L^{p,\kappa}_{\rho,\infty}(\mathbb H^n):=\bigcup_{0<\theta<\infty}L^{p,\kappa}_{\rho,\theta}(\mathbb H^n).
\end{equation*}
\end{defin}

\begin{defin}
Let $\rho$ be the auxiliary function determined by $V\in RH_q$ with $q\geq Q/2$. Let $p=1$ and $0\leq\kappa<1$. For given $0<\theta<\infty$, the weak Morrey space $WL^{1,\kappa}_{\rho,\theta}(\mathbb H^n)$ is defined to be the set of all measurable functions $f$ on $\mathbb H^n$ such that
\begin{equation*}
\frac{1}{|B(u_0,r)|^{\kappa}}\sup_{\lambda>0}\lambda\cdot\big|\big\{u\in B(u_0,r):|f(u)|>\lambda\big\}\big|
\leq C\cdot\left[1+\frac{r}{\rho(u_0)}\right]^{\theta}
\end{equation*}
holds for every ball $B(u_0,r)$ in $\mathbb H^n$, or equivalently,
\begin{equation*}
\big\|f\big\|_{WL^{1,\kappa}_{\rho,\theta}(\mathbb H^n)}:=\sup_{B(u_0,r)}\left[1+\frac{r}{\rho(u_0)}\right]^{-\theta}\frac{1}{|B(u_0,r)|^{\kappa}}
\sup_{\lambda>0}\lambda\cdot\big|\big\{u\in B(u_0,r):|f(u)|>\lambda\big\}\big|<\infty.
\end{equation*}
Correspondingly, we define
\begin{equation*}
WL^{1,\kappa}_{\rho,\infty}(\mathbb H^n):=\bigcup_{0<\theta<\infty}WL^{1,\kappa}_{\rho,\theta}(\mathbb H^n).
\end{equation*}
\end{defin}
\begin{rem}
$(i)$ Obviously, if we take $\theta=0$ or $V\equiv0$, then this Morrey space $L^{p,\kappa}_{\rho,\theta}(\mathbb H^n)$ (or weak Morrey space $WL^{1,\kappa}_{\rho,\theta}(\mathbb H^n)$) is just the Morrey space $L^{p,\kappa}(\mathbb H^n)$ (or weak Morrey space $WL^{1,\kappa}(\mathbb H^n)$), which was defined and studied by Guliyev et al. \cite{guliyev}.

$(ii)$ According to the above definitions, one has
\begin{align}
L^{p,\kappa}(\mathbb H^n)\subset L^{p,\kappa}_{\rho,\theta_1}(\mathbb H^n)\subset L^{p,\kappa}_{\rho,\theta_2}(\mathbb H^n);& \label{22}\\
WL^{1,\kappa}(\mathbb H^n)\subset WL^{1,\kappa}_{\rho,\theta_1}(\mathbb H^n)\subset WL^{1,\kappa}_{\rho,\theta_2}(\mathbb H^n),& \label{23}
\end{align}
whenever $0<\theta_1<\theta_2<\infty$. Hence,
\begin{equation*}
L^{p,\kappa}(\mathbb H^n)\subset L^{p,\kappa}_{\rho,\infty}(\mathbb H^n)\quad \mathrm{and} \quad WL^{1,\kappa}(\mathbb H^n)\subset WL^{1,\kappa}_{\rho,\infty}(\mathbb H^n),
\end{equation*}
for all $(p,\kappa)\in[1,\infty)\times[0,1)$.

$(iii)$ We can define a norm on the space $L^{p,\kappa}_{\rho,\infty}(\mathbb H^n)$, which makes it into a Banach space. In view of \eqref{22}, for any given $f\in L^{p,\kappa}_{\rho,\infty}(\mathbb H^n)$, let
\begin{equation*}
\theta^*:=\inf\big\{\theta>0:f\in L^{p,\kappa}_{\rho,\theta}(\mathbb H^n)\big\}.
\end{equation*}
Now define the functional $\|\cdot\|_{\star}$ by
\begin{equation}\label{newnorm}
\|f\|_{\star}=\big\|f\big\|_{L^{p,\kappa}_{\rho,\infty}(\mathbb H^n)}:=\big\|f\big\|_{L^{p,\kappa}_{\rho,\theta^*}(\mathbb H^n)}.
\end{equation}
It is easy to check that this functional $\|\cdot\|_{\star}$ satisfies the axioms of a norm; i.e., that for $f,g\in L^{p,\kappa}_{\rho,\infty}(\mathbb H^n)$ and $\lambda\in\mathbb R$,
\begin{itemize}
  \item it is positive definite: $\|f\|_{\star}\geq0$, and $\|f\|_{\star}=0\Leftrightarrow f=0;$
  \item it is multiplicative: $\|\lambda f\|_{\star}=|\lambda|\|f\|_{\star};$
  \item it satisfies the triangle inequality: $\|f+g\|_{\star}\leq\|f\|_{\star}+\|g\|_{\star}$.
\end{itemize}

$(iv)$ In view of \eqref{23}, for any given $f\in WL^{1,\kappa}_{\rho,\infty}(\mathbb H^n)$, let
\begin{equation*}
\theta^{**}:=\inf\big\{\theta>0:f\in WL^{1,\kappa}_{\rho,\theta}(\mathbb H^n)\big\}.
\end{equation*}
Similarly, we define the functional $\|\cdot\|_{\star\star}$ by
\begin{equation*}
\|f\|_{\star\star}=\big\|f\big\|_{WL^{1,\kappa}_{\rho,\infty}(\mathbb H^n)}:=\big\|f\big\|_{WL^{1,\kappa}_{\rho,\theta^{**}}(\mathbb H^n)}.
\end{equation*}
By definition, we can easily show that this functional $\|\cdot\|_{\star\star}$ satisfies the axioms of a (quasi-)norm, and $WL^{1,\kappa}_{\rho,\infty}(\mathbb H^n)$ is a (quasi-)normed linear space.
\end{rem}

Since Morrey space $L^{p,\kappa}_{\rho,\theta}(\mathbb H^n)$ (or weak Morrey space $WL^{1,\kappa}_{\rho,\theta}(\mathbb H^n)$) could be viewed as an extension of Lebesgue (or weak Lebesgue) space on $\mathbb H^n$ (when $\kappa=\theta=0$, or $\kappa=0$,$V\equiv0$), it is accordingly natural to investigate the boundedness properties of the Littlewood-Paley functions in the framework of
Morrey spaces. In this article we will extend Theorems \ref{strongg} and \ref{strongs} to the Morrey spaces on $\mathbb H^n$.

\section{Boundedness of the Littlewood-Paley function $\mathfrak{g}_{\mathcal L}$}\label{sec3}
In this section, we will establish the boundedness properties of the Littlewood-Paley function $\mathfrak{g}_{\mathcal L}$ acting on $L^{p,\kappa}_{\rho,\infty}(\mathbb H^n)$ for $(p,\kappa)\in[1,\infty)\times(0,1)$. Recall the Littlewood-Paley function $\mathfrak{g}_{\mathcal L}$ defined in the introduction by
\begin{equation*}
\mathfrak{g}_{\mathcal L}(f)(u)=\bigg(\int_0^{\infty}\big|(s\mathcal L)e^{-s\mathcal L}f(u)\big|^2\frac{ds}{s}\bigg)^{1/2},\quad u\in\mathbb H^n,
\end{equation*}
where $\big\{e^{-s\mathcal L}:s>0\big\}$ is the semigroup generated by $\mathcal L$. Let $\mathcal Q_{s}(u,v)$ denote the kernel of $(s\mathcal L)e^{-s\mathcal L},s>0$. Then we have
\begin{equation}\label{kauv}
\mathfrak{g}_{\mathcal L}(f)(u)=\bigg(\int_0^{\infty}\bigg|\int_{\mathbb H^n}\mathcal Q_{s}(u,v)f(v)\,dv\bigg|^2\frac{ds}{s}\bigg)^{1/2}.
\end{equation}

We now present our main results as follows.
\begin{thm}\label{mainthm:1}
Let $\rho$ be as in \eqref{rho}. Let $1<p<\infty$ and $0<\kappa<1$. If $V\in RH_q$ with $q\geq Q/2$, then the Littlewood-Paley operator $\mathfrak{g}_{\mathcal L}$ is bounded on $L^{p,\kappa}_{\rho,\infty}(\mathbb H^n)$.
\end{thm}

\begin{thm}\label{mainthm:2}
Let $\rho$ be as in \eqref{rho}. Let $p=1$ and $0<\kappa<1$. If $V\in RH_q$ with $q\geq Q/2$, then the Littlewood-Paley operator $\mathfrak{g}_{\mathcal L}$ is bounded from $L^{1,\kappa}_{\rho,\infty}(\mathbb H^n)$ to $WL^{1,\kappa}_{\rho,\infty}(\mathbb H^n)$.
\end{thm}

We need the following lemma which establishes the estimate of the kernel $\mathcal Q_s(u,v)$ related to the nonnegative potential $V$ and plays a key role in the proofs of our main results. Its proof (based on Lemma \ref{ker1}) can be found in \cite{lin}.
\begin{lem}[\cite{lin}]\label{kernel}
Let $V\in RH_q$ with $q\geq Q/2$, and let $\rho(u)$ be the auxiliary function determined by $V$. For every positive integer $N\in\mathbb N$, there exists a positive constant $C_{N}>0$ such that, for any $u$ and $v$ in $\mathbb H^n$,
\begin{equation}\label{WH1}
\big|\mathcal Q_{s}(u,v)\big|\leq C_N\cdot s^{-Q/2}\exp\bigg(-\frac{|v^{-1}u|^2}{As}\bigg)\bigg[1+\frac{\sqrt{s\,}}{\rho(u)}+\frac{\sqrt{s\,}}{\rho(v)}\bigg]^{-N},\quad s>0.
\end{equation}
\end{lem}

We are now ready to show our main theorems.
\begin{proof}[Proof of Theorem $\ref{mainthm:1}$]
For any given $f\in L^{p,\kappa}_{\rho,\infty}(\mathbb H^n)$ with $1\leq p<\infty$ and $0<\kappa<1$, suppose that $f\in L^{p,\kappa}_{\rho,\theta^*}(\mathbb H^n)$ for some $\theta^*>0$, where
\begin{equation*}
\theta^*=\inf\big\{\theta>0:f\in L^{p,\kappa}_{\rho,\theta}(\mathbb H^n)\big\}\quad\mathrm{and}\quad\big\|f\big\|_{L^{p,\kappa}_{\rho,\infty}(\mathbb H^n)}=\big\|f\big\|_{L^{p,\kappa}_{\rho,\theta^*}(\mathbb H^n)}.
\end{equation*}
By definition, we only need to show that for each fixed ball $B=B(u_0,r)$ of $\mathbb H^n$, there is some $\vartheta>0$ such that
\begin{equation}\label{Main1}
\bigg(\frac{1}{|B(u_0,r)|^{\kappa}}\int_{B(u_0,r)}\big|\mathfrak{g}_{\mathcal L}(f)(u)\big|^p\,du\bigg)^{1/p}\lesssim\left[1+\frac{r}{\rho(u_0)}\right]^{\vartheta}
\end{equation}
holds true for $f\in L^{p,\kappa}_{\rho,\theta^*}(\mathbb H^n)$ with $(p,\kappa)\in(1,\infty)\times(0,1)$. Using the standard technique, we decompose the function $f$ as
\begin{equation*}
\begin{cases}
f=f_1+f_2\in L^{p,\kappa}_{\rho,\theta^*}(\mathbb H^n);\  &\\
f_1=f\cdot\chi_{2B};\  &\\
f_2=f\cdot\chi_{(2B)^{\complement}},
\end{cases}
\end{equation*}
where $2B$ is the ball centered at $u_0$ of radius $2r$, $(2B)^{\complement}=\mathbb H^n\backslash(2B)$ denotes its complement and $\chi_{E}$ denotes the characteristic function of the set $E$. Then by the sublinearity of $\mathfrak{g}_{\mathcal L}$, we write
\begin{equation*}
\begin{split}
\bigg(\frac{1}{|B(u_0,r)|^{\kappa}}\int_{B(u_0,r)}\big|\mathfrak{g}_{\mathcal L}(f)(u)\big|^p\,du\bigg)^{1/p}
&\leq\bigg(\frac{1}{|B(u_0,r)|^{\kappa}}\int_{B(u_0,r)}\big|\mathfrak{g}_{\mathcal L}(f_1)(u)\big|^p\,du\bigg)^{1/p}\\
&+\bigg(\frac{1}{|B(u_0,r)|^{\kappa}}\int_{B(u_0,r)}\big|\mathfrak{g}_{\mathcal L}(f_2)(u)\big|^p\,du\bigg)^{1/p}\\
&:=I_1+I_2.
\end{split}
\end{equation*}
In what follows, we consider each part separately. For the first term $I_1$, by Theorem \ref{strongg} (1), we have
\begin{equation*}
\begin{split}
I_1&=\bigg(\frac{1}{|B(u_0,r)|^{\kappa}}\int_{B(u_0,r)}\big|\mathfrak{g}_{\mathcal L}(f_1)(u)\big|^p\,du\bigg)^{1/p}\\
&\leq C\cdot\frac{1}{|B|^{\kappa/p}}\bigg(\int_{\mathbb H^n}\big|f_1(u)\big|^p\,du\bigg)^{1/p}\\
&=C\cdot\frac{1}{|B|^{\kappa/p}}\bigg(\int_{2B}\big|f(u)\big|^p\,du\bigg)^{1/p}\\
&\leq C\big\|f\big\|_{L^{p,\kappa}_{\rho,\theta^*}(\mathbb H^n)}\cdot
\frac{|2B|^{\kappa/p}}{|B|^{\kappa/p}}\cdot\left[1+\frac{2r}{\rho(u_0)}\right]^{\theta^*}.
\end{split}
\end{equation*}
Moreover, observe that for any fixed $\theta^*>0$,
\begin{equation}\label{2rx}
1\leq\left[1+\frac{2r}{\rho(u_0)}\right]^{\theta^*}\leq 2^{\theta^*}\left[1+\frac{r}{\rho(u_0)}\right]^{\theta^*},
\end{equation}
which further implies that
\begin{equation*}
\begin{split}
I_1&\leq C_{\theta^*,n}\big\|f\big\|_{L^{p,\kappa}_{\rho,\infty}(\mathbb H^n)}\left[1+\frac{r}{\rho(u_0)}\right]^{\theta^*}.
\end{split}
\end{equation*}
Next we estimate the other term $I_2$. We first claim that the following inequality
\begin{equation}\label{keyWH}
\mathfrak{g}_{\mathcal L}(f_2)(u)\lesssim\int_{(2B)^{\complement}}\bigg[1+\frac{|v^{-1}u|}{\rho(u)}\bigg]^{-N}\frac{1}{|v^{-1}u|^{Q}}\cdot|f(v)|\,dv
\end{equation}
holds for any $u\in B(u_0,r)$. In fact, from \eqref{kauv} and Lemma \ref{kernel}, it follows that
\begin{equation*}
\begin{split}
\mathfrak{g}_{\mathcal L}(f_2)(u)&\leq\bigg(\int_0^{\infty}\bigg|\int_{\mathbb H^n}\frac{C_N}{s^{Q/2}}\cdot\exp\bigg(-\frac{|v^{-1}u|^2}{As}\bigg)
\bigg[1+\frac{\sqrt{s}}{\rho(u)}\bigg]^{-N}|f_2(v)|\,dv\bigg|^2\frac{ds}{s}\bigg)^{1/2}\\
&=\bigg(\int_{|v^{-1}u|^2}^{\infty}\bigg|\int_{\mathbb H^n}\frac{C_N}{s^{Q/2}}\cdot\exp\bigg(-\frac{|v^{-1}u|^2}{As}\bigg)
\bigg[1+\frac{\sqrt{s}}{\rho(u)}\bigg]^{-N}|f_2(v)|\,dv\bigg|^2\frac{ds}{s}\bigg)^{1/2}\\
&+\bigg(\int_0^{|v^{-1}u|^2}\bigg|\int_{\mathbb H^n}\frac{C_N}{s^{Q/2}}\cdot\exp\bigg(-\frac{|v^{-1}u|^2}{As}\bigg)
\bigg[1+\frac{\sqrt{s}}{\rho(u)}\bigg]^{-N}|f_2(v)|\,dv\bigg|^2\frac{ds}{s}\bigg)^{1/2}\\
&:=\mathfrak{g}_{\mathcal L}^{\infty}(f_2)(u)+\mathfrak{g}_{\mathcal L}^{0}(f_2)(u).
\end{split}
\end{equation*}
When $s>|v^{-1}u|^2$, then $\sqrt{s\,}>|v^{-1}u|$, and hence
\begin{equation*}
\begin{split}
\mathfrak{g}_{\mathcal L}^{\infty}(f_2)(u)&\leq\bigg(\int_{|v^{-1}u|^2}^{\infty}\bigg|\int_{\mathbb H^n}\frac{C_N}{s^{Q/2}}\cdot
\bigg[1+\frac{\sqrt{s}}{\rho(u)}\bigg]^{-N}|f_2(v)|\,dv\bigg|^2\frac{ds}{s}\bigg)^{1/2}\\
&\leq\bigg(\int_{|v^{-1}u|^2}^{\infty}\bigg|\int_{\mathbb H^n}\frac{C_N}{s^{Q/2}}\cdot
\bigg[1+\frac{|v^{-1}u|}{\rho(u)}\bigg]^{-N}|f_2(v)|\,dv\bigg|^2\frac{ds}{s}\bigg)^{1/2}\\
&\leq\int_{\mathbb H^n}C_N\cdot
\bigg[1+\frac{|v^{-1}u|}{\rho(u)}\bigg]^{-N}|f_2(v)|\bigg(\int_{|v^{-1}u|^2}^{\infty}\frac{ds}{s^{Q+1}}\bigg)^{1/2}dv\\
&\leq C_{N,n}\int_{(2B)^{\complement}}\bigg[1+\frac{|v^{-1}u|}{\rho(u)}\bigg]^{-N}\frac{1}{|v^{-1}u|^{Q}}\cdot|f(v)|\,dv,
\end{split}
\end{equation*}
where in the third step we have used the Minkowski inequality. On the other hand, a trivial computation leads to that
\begin{equation*}
\begin{split}
&\int_{\mathbb H^n}\frac{C_N}{s^{Q/2}}\cdot\exp\bigg(-\frac{|v^{-1}u|^2}{As}\bigg)
\bigg[1+\frac{\sqrt{s}}{\rho(u)}\bigg]^{-N}|f_2(v)|\,dv\\
&\leq\int_{\mathbb H^n}\frac{C_{N,A}}{s^{Q/2}}\cdot\bigg(\frac{|v^{-1}u|^2}{s}\bigg)^{-(Q/2+N/2+\gamma/2)}
\bigg[1+\frac{\sqrt{s}}{\rho(u)}\bigg]^{-N}|f_2(v)|\,dv\\
&=\int_{\mathbb H^n}\frac{C_{N,A}}{|v^{-1}u|^{Q}}\cdot\bigg(\frac{s}{|v^{-1}u|^2}\bigg)^{\gamma/2}\bigg(\frac{\sqrt{s}}{|v^{-1}u|}\bigg)^{N}
\bigg[1+\frac{\sqrt{s}}{\rho(u)}\bigg]^{-N}|f_2(v)|\,dv,
\end{split}
\end{equation*}
where $\gamma>0$ is a positive constant. It is easy to check that when $0\leq s\leq|v^{-1}u|^2$, one has
\begin{equation*}
\frac{\sqrt{s\,}}{|v^{-1}u|}\leq\frac{\sqrt{s\,}+\rho(u)}{|v^{-1}u|+\rho(u)}.
\end{equation*}
Hence,
\begin{equation*}
\begin{split}
&\int_{\mathbb H^n}\frac{C_N}{s^{Q/2}}\cdot\exp\bigg(-\frac{|v^{-1}u|^2}{As}\bigg)\bigg[1+\frac{\sqrt{s}}{\rho(u)}\bigg]^{-N}|f_2(v)|\,dv\\
&\leq\int_{\mathbb H^n}\frac{C_{N,A}}{|v^{-1}u|^{Q}}\cdot
\bigg(\frac{s}{|v^{-1}u|^2}\bigg)^{\gamma/2}\bigg[\frac{\sqrt{s}+\rho(u)}{|v^{-1}u|+\rho(u)}\bigg]^{N}
\bigg[\frac{\sqrt{s}+\rho(u)}{\rho(u)}\bigg]^{-N}|f_2(v)|\,dv\\
&=\int_{\mathbb H^n}\frac{C_{N,A}}{|v^{-1}u|^{Q}}\cdot\bigg(\frac{s}{|v^{-1}u|^2}\bigg)^{\gamma/2}
\bigg[1+\frac{|v^{-1}u|}{\rho(u)}\bigg]^{-N}|f_2(v)|\,dv.
\end{split}
\end{equation*}
This, together with Minkowski's inequality for integrals, shows that
\begin{equation*}
\begin{split}
\mathfrak{g}_{\mathcal L}^{0}(f_2)(u)&\leq\bigg(\int_0^{|v^{-1}u|^2}\bigg|\int_{\mathbb H^n}\frac{C_{N,A}}{|v^{-1}u|^{Q}}\cdot\bigg(\frac{s}{|v^{-1}u|^2}\bigg)^{\gamma/2}
\bigg[1+\frac{|v^{-1}u|}{\rho(u)}\bigg]^{-N}|f_2(v)|\,dv\bigg|^2\frac{ds}{s}\bigg)^{1/2}\\
&\leq\int_{\mathbb H^n}\frac{C_{N,A}}{|v^{-1}u|^{Q}}\cdot
\bigg[1+\frac{|v^{-1}u|}{\rho(u)}\bigg]^{-N}|f_2(v)|\bigg\{\int_0^{|v^{-1}u|^2}\bigg(\frac{s}{|v^{-1}u|^2}\bigg)^{\gamma}\frac{ds}{s}\bigg\}^{1/2}dv\\
&\leq C_{N,A,\gamma}\int_{(2B)^{\complement}}\bigg[1+\frac{|v^{-1}u|}{\rho(u)}\bigg]^{-N}\frac{1}{|v^{-1}u|^{Q}}\cdot|f(v)|\,dv.
\end{split}
\end{equation*}
Combining the above two estimates produces the desired inequality \eqref{keyWH} for any $u\in B(u_0,r)$. Notice that for any $u\in B(u_0,r)$ and $v\in (2B)^{\complement}$, one has
\begin{equation*}
\big|v^{-1}u\big|=\big|(v^{-1}u_0)\cdot(u_0^{-1}u)\big|\leq\big|v^{-1}u_0\big|+\big|u_0^{-1}u\big|
\end{equation*}
and
\begin{equation*}
\big|v^{-1}u\big|=\big|(v^{-1}u_0)\cdot(u_0^{-1}u)\big|\geq\big|v^{-1}u_0\big|-\big|u_0^{-1}u\big|.
\end{equation*}
Thus,
\begin{equation*}
\frac{1}{\,2\,}\big|v^{-1}u_0\big|\leq\big|v^{-1}u\big|\leq\frac{3}{\,2\,}\big|v^{-1}u_0\big|,
\end{equation*}
i.e., $|v^{-1}u|\approx|v^{-1}u_0|$. This fact, along with \eqref{keyWH}, implies that for any $u\in B(u_0,r)$,
\begin{equation}\label{Talpha}
\begin{split}
\big|\mathfrak{g}_{\mathcal L}(f_2)(u)\big|
&\leq C\int_{(2B)^{\complement}}\bigg[1+\frac{|v^{-1}u_0|}{\rho(u)}\bigg]^{-N}\frac{1}{|v^{-1}u_0|^{Q}}\cdot|f(v)|\,dv\\
&=C\sum_{k=1}^\infty\int_{2^kr\leq|v^{-1}u_0|<2^{k+1}r}\bigg[1+\frac{|v^{-1}u_0|}{\rho(u)}\bigg]^{-N}
\frac{1}{|v^{-1}u_0|^{Q}}\cdot|f(v)|\,dv\\
&\leq C\sum_{k=1}^\infty\frac{1}{|B(u_0,2^{k+1}r)|}
\int_{|v^{-1}u_0|<2^{k+1}r}\bigg[1+\frac{2^kr}{\rho(u)}\bigg]^{-N}|f(v)|\,dv.
\end{split}
\end{equation}
In view of \eqref{com2} and \eqref{2rx}, we can further obtain
\begin{align}\label{Tf2}
\big|\mathfrak{g}_{\mathcal L}(f_2)(u)\big|
&\leq C\sum_{k=1}^\infty\frac{1}{|B(u_0,2^{k+1}r)|}\notag\\
&\times\int_{|v^{-1}u_0|<2^{k+1}r}\left[1+\frac{r}{\rho(u_0)}\right]^{N\cdot\frac{N_0}{N_0+1}}
\left[1+\frac{2^kr}{\rho(u_0)}\right]^{-N}|f(v)|\,dv\notag\\
&\leq C\sum_{k=1}^\infty\frac{1}{|B(u_0,2^{k+1}r)|}\notag\\
&\times\int_{B(u_0,2^{k+1}r)}\left[1+\frac{r}{\rho(u_0)}\right]^{N\cdot\frac{N_0}{N_0+1}}
\left[1+\frac{2^{k+1}r}{\rho(u_0)}\right]^{-N}|f(v)|\,dv.
\end{align}
We consider each term in the sum of \eqref{Tf2} separately. By using H\"older's inequality, we can deduce that for each fixed $k\in\mathbb N$,
\begin{equation*}
\begin{split}
&\frac{1}{|B(u_0,2^{k+1}r)|}\int_{B(u_0,2^{k+1}r)}\big|f(v)\big|\,dv\\
&\leq\frac{1}{|B(u_0,2^{k+1}r)|}\bigg(\int_{B(u_0,2^{k+1}r)}\big|f(v)\big|^p\,dv\bigg)^{1/p}
\bigg(\int_{B(u_0,2^{k+1}r)}1\,dv\bigg)^{1/{p'}}\\
&\leq C\big\|f\big\|_{L^{p,\kappa}_{\rho,\theta^*}(\mathbb H^n)}\cdot\frac{|B(u_0,2^{k+1}r)|^{{\kappa}/p}}{|B(u_0,2^{k+1}r)|^{1/p}}
\left[1+\frac{2^{k+1}r}{\rho(u_0)}\right]^{\theta^*}.
\end{split}
\end{equation*}
This allows us to obtain
\begin{equation*}
\begin{split}
I_2&\leq C\big\|f\big\|_{L^{p,\kappa}_{\rho,\theta^*}(\mathbb H^n)}\cdot\frac{|B(u_0,r)|^{1/p}}{|B(u_0,r)|^{{\kappa}/p}}
\sum_{k=1}^\infty\frac{|B(u_0,2^{k+1}r)|^{{\kappa}/p}}{|B(u_0,2^{k+1}r)|^{1/p}}
\left[1+\frac{r}{\rho(u_0)}\right]^{N\cdot\frac{N_0}{N_0+1}}\left[1+\frac{2^{k+1}r}{\rho(u_0)}\right]^{-N+\theta^*}\\
&=C\big\|f\big\|_{L^{p,\kappa}_{\rho,\theta^*}(\mathbb H^n)}\left[1+\frac{r}{\rho(u_0)}\right]^{N\cdot\frac{N_0}{N_0+1}}
\sum_{k=1}^\infty\frac{|B(u_0,r)|^{{(1-\kappa)}/p}}{|B(u_0,2^{k+1}r)|^{{(1-\kappa)}/p}}
\left[1+\frac{2^{k+1}r}{\rho(u_0)}\right]^{-N+\theta^*}.
\end{split}
\end{equation*}
Consequently, by choosing $N$ large enough such that $N\geq\theta^*$, and the last series is convergent. Then we have
\begin{equation*}
\begin{split}
I_2&\leq C\big\|f\big\|_{L^{p,\kappa}_{\rho,\infty}(\mathbb H^n)}
\left[1+\frac{r}{\rho(u_0)}\right]^{N\cdot\frac{N_0}{N_0+1}}\sum_{k=1}^\infty\left(\frac{|B(u_0,r)|}{|B(u_0,2^{k+1}r)|}\right)^{{(1-\kappa)}/p}\\
&\leq C\big\|f\big\|_{L^{p,\kappa}_{\rho,\infty}(\mathbb H^n)}
\left[1+\frac{r}{\rho(u_0)}\right]^{N\cdot\frac{N_0}{N_0+1}},
\end{split}
\end{equation*}
where the last inequality follows from the fact that $1-\kappa>0$. Summing up the above estimates for $I_1$ and $I_2$, and letting $\vartheta=\max\big\{\theta^*,N\cdot\frac{N_0}{N_0+1}\big\}$, with $N\geq\theta^*$, we obtain the desired inequality \eqref{Main1}. This concludes the proof of Theorem \ref{mainthm:1}.
\end{proof}

\begin{proof}[Proof of Theorem $\ref{mainthm:2}$]
According to the definition, it suffices to prove that for each given ball $B=B(u_0,r)$ of $\mathbb H^n$, there is some $\vartheta>0$ such that
\begin{equation}\label{Main2}
\frac{1}{|B(u_0,r)|^{\kappa}}\sup_{\lambda>0}\lambda\cdot\big|\big\{u\in B(u_0,r):|\mathfrak{g}_{\mathcal L}(f)(u)|>\lambda\big\}\big|
\lesssim\left[1+\frac{r}{\rho(u_0)}\right]^{\vartheta}
\end{equation}
holds true for given $f\in L^{1,\kappa}_{\rho,\theta^*}(\mathbb H^n)$ with some $\theta^*>0$ and $0<\kappa<1$. We decompose the function $f$ as
\begin{equation*}
\begin{cases}
f=f_1+f_2\in L^{1,\kappa}_{\rho,\theta^*}(\mathbb H^n);\  &\\
f_1=f\cdot\chi_{2B};\  &\\
f_2=f\cdot\chi_{(2B)^{\complement}}.
\end{cases}
\end{equation*}
Then for any given $\lambda>0$, we can write
\begin{equation*}
\begin{split}
&\frac{1}{|B(u_0,r)|^{\kappa}}\lambda\cdot\big|\big\{u\in B(u_0,r):|\mathfrak{g}_{\mathcal L}(f)(u)|>\lambda\big\}\big|\\
&\leq\frac{1}{|B(u_0,r)|^{\kappa}}\lambda\cdot\big|\big\{u\in B(u_0,r):|\mathfrak{g}_{\mathcal L}(f_1)(u)|>\lambda/2\big\}\big|\\
&+\frac{1}{|B(u_0,r)|^{\kappa}}\lambda\cdot\big|\big\{u\in B(u_0,r):|\mathfrak{g}_{\mathcal L}(f_2)(u)|>\lambda/2\big\}\big|\\
&:=J_1+J_2.
\end{split}
\end{equation*}
We first give the estimate for the term $J_1$. By Theorem \ref{strongg} (2), we get
\begin{equation*}
\begin{split}
J_1&=\frac{1}{|B(u_0,r)|^{\kappa}}\lambda\cdot\big|\big\{u\in B(u_0,r):|\mathfrak{g}_{\mathcal L}(f_1)(u)|>\lambda/2\big\}\big|\\
&\leq C\cdot\frac{1}{|B|^{\kappa}}\bigg(\int_{\mathbb H^n}\big|f_1(u)\big|\,du\bigg)\\
&=C\cdot\frac{1}{|B|^{\kappa}}\bigg(\int_{2B}\big|f(u)\big|\,du\bigg)\\
&\leq C\big\|f\big\|_{L^{1,\kappa}_{\rho,\theta^*}(\mathbb H^n)}\cdot\frac{|2B|^{\kappa}}{|B|^{\kappa}}\left[1+\frac{2r}{\rho(u_0)}\right]^{\theta^*}.
\end{split}
\end{equation*}
Therefore, in view of \eqref{2rx}, we have
\begin{equation*}
J_1\leq C_{\theta^*,n}\big\|f\big\|_{L^{1,\kappa}_{\rho,\infty}(\mathbb H^n)}\left[1+\frac{r}{\rho(u_0)}\right]^{\theta^*}.
\end{equation*}
As for the second term $J_2$, by using the pointwise inequality \eqref{Tf2} and Chebyshev's inequality, we can deduce that
\begin{equation}\label{Tf2pr}
\begin{split}
J_2&=\frac{1}{|B(u_0,r)|^{\kappa}}\lambda\cdot\big|\big\{u\in B(u_0,r):|\mathfrak{g}_{\mathcal L}(f_2)(u)|>\lambda/2\big\}\big|\\
&\leq\frac{2}{|B(u_0,r)|^{\kappa}}\bigg(\int_{B(u_0,r)}\big|\mathfrak{g}_{\mathcal L}(f_2)(u)\big|\,du\bigg)\\
&\leq C\cdot\frac{|B(u_0,r)|}{|B(u_0,r)|^{\kappa}}
\sum_{k=1}^\infty\frac{1}{|B(u_0,2^{k+1}r)|}\\
&\times\int_{B(u_0,2^{k+1}r)}\left[1+\frac{r}{\rho(u_0)}\right]^{N\cdot\frac{N_0}{N_0+1}}
\left[1+\frac{2^{k+1}r}{\rho(u_0)}\right]^{-N}|f(v)|\,dv.
\end{split}
\end{equation}
We consider each term in the sum of \eqref{Tf2pr} separately. For each fixed $k\in\mathbb N$, we have
\begin{equation*}
\begin{split}
&\frac{1}{|B(u_0,2^{k+1}r)|}\int_{B(u_0,2^{k+1}r)}|f(v)|\,dv\\
&\leq C\big\|f\big\|_{L^{1,\kappa}_{\rho,\theta^*}(\mathbb H^n)}\cdot
\frac{|B(u_0,2^{k+1}r)|^{\kappa}}{|B(u_0,2^{k+1}r)|}\left[1+\frac{2^{k+1}r}{\rho(u_0)}\right]^{\theta^*}.
\end{split}
\end{equation*}
Consequently,
\begin{equation*}
\begin{split}
J_2&\leq C\big\|f\big\|_{L^{1,\kappa}_{\rho,\theta^*}(\mathbb H^n)}
\cdot\frac{|B(u_0,r)|}{|B(u_0,r)|^{\kappa}}\sum_{k=1}^\infty\frac{|B(u_0,2^{k+1}r)|^{\kappa}}{|B(u_0,2^{k+1}r)|}
\left[1+\frac{r}{\rho(u_0)}\right]^{N\cdot\frac{N_0}{N_0+1}}\left[1+\frac{2^{k+1}r}{\rho(u_0)}\right]^{-N+\theta^*}\\
&=C\big\|f\big\|_{L^{1,\kappa}_{\rho,\theta^*}(\mathbb H^n)}
\left[1+\frac{r}{\rho(u_0)}\right]^{N\cdot\frac{N_0}{N_0+1}}\sum_{k=1}^\infty\frac{|B(u_0,r)|^{{1-\kappa}}}{|B(u_0,2^{k+1}r)|^{{1-\kappa}}}
\left[1+\frac{2^{k+1}r}{\rho(u_0)}\right]^{-N+\theta^*}.
\end{split}
\end{equation*}
Therefore, by selecting $N$ large enough such that $N\geq\theta^*$, we thus have
\begin{equation*}
\begin{split}
J_2&\leq C\big\|f\big\|_{L^{1,\kappa}_{\rho,\infty}(\mathbb H^n)}\left[1+\frac{r}{\rho(u_0)}\right]^{N\cdot\frac{N_0}{N_0+1}}
\sum_{k=1}^\infty\left(\frac{|B(u_0,r)|}{|B(u_0,2^{k+1}r)|}\right)^{{(1-\kappa)}}\\
&\leq C\big\|f\big\|_{L^{1,\kappa}_{\rho,\infty}(\mathbb H^n)}
\left[1+\frac{r}{\rho(u_0)}\right]^{N\cdot\frac{N_0}{N_0+1}},
\end{split}
\end{equation*}
where the last inequality holds because $0<\kappa<1$. Now we choose $\vartheta=\max\big\{\theta^*,N\cdot\frac{N_0}{N_0+1}\big\}$ with $N\geq\theta^*$. Summing up the above estimates for $J_1$ and $J_2$, and then taking the supremum over all $\lambda>0$, we obtain the desired inequality \eqref{Main2}. This concludes the proof of Theorem \ref{mainthm:2}.
\end{proof}

We also consider the Littlewood-Paley function with respect to the Poisson semigroup $\big\{e^{-s\sqrt{\mathcal L}}\big\}_{s>0}$, which is defined by
\begin{align}\label{gp}
\mathfrak{g}_{\sqrt{\mathcal L}}(f)(u)&:=\bigg(\int_0^{\infty}\Big|s\cdot\frac{d}{ds}e^{-s\sqrt{\mathcal L}}f(u)\Big|^2\frac{ds}{s}\bigg)^{1/2}\notag\\
&=\bigg(\int_0^{\infty}\big|(s\sqrt{\mathcal L})e^{-s\sqrt{\mathcal L}}f(u)\big|^2\frac{ds}{s}\bigg)^{1/2}.
\end{align}
As illustrated below, this integral operator $\mathfrak{g}_{\sqrt{\mathcal L}}$ will be dominated by $\mathfrak{g}_{\mathcal L}$. To this end, we first recall the subordination formula
\begin{equation}\label{subor}
e^{-s\sqrt{\mathcal L}}f(u)=\frac{1}{\sqrt{\pi}}\int_0^{\infty}\frac{e^{-v}}{\sqrt{v}}\cdot e^{-\frac{s^2}{4v}\mathcal L}f(u)\,dv,\quad u\in\mathbb H^n.
\end{equation}
This allows us to obtain
\begin{equation*}
\begin{split}
\big|(s\sqrt{\mathcal L})e^{-s\sqrt{\mathcal L}}f(u)\big|
&=\bigg|s\cdot\frac{d}{ds}e^{-s\sqrt{\mathcal L}}f(u)\bigg|\\
&=\bigg|\frac{2}{\sqrt{\pi}}\int_0^{\infty}\frac{e^{-v}}{\sqrt{v}}\cdot\Big(\frac{s^2}{4v}\mathcal L\Big)e^{-\frac{s^2}{4v}\mathcal L}f(u)\,dv\bigg|.
\end{split}
\end{equation*}
From this, it follows that for all $u\in\mathbb H^n$,
\begin{equation*}
\begin{split}
\mathfrak{g}_{\sqrt{\mathcal L}}(f)(u)&\leq\frac{2}{\sqrt{\pi}}\int_0^{\infty}\frac{e^{-v}}{\sqrt{v}}\cdot
\bigg(\int_0^{\infty}\Big|\Big(\frac{s^2}{4v}\mathcal L\Big)e^{-\frac{s^2}{4v}\mathcal L}f(u)\Big|^2\frac{ds}{s}\bigg)^{1/2}dv\\
&=\frac{2}{\sqrt{\pi}}\int_0^{\infty}\frac{e^{-v}}{\sqrt{v}}\cdot
\bigg(\int_0^{\infty}\big|(t\mathcal L)e^{-t\mathcal L}f(u)\big|^2\frac{dt}{2t}\bigg)^{1/2}dv\\
&=\mathfrak{g}_{\mathcal L}(f)(u)\cdot\frac{\sqrt{2}}{\sqrt{\pi}}\int_0^{\infty}\frac{e^{-v}}{\sqrt{v}}\,dv=\sqrt{2}\cdot\mathfrak{g}_{\mathcal L}(f)(u).
\end{split}
\end{equation*}
Hence, we know that under the conditions of Theorems \ref{mainthm:1} and \ref{mainthm:2}, the conclusions also hold for the operator $\mathfrak{g}_{\sqrt{\mathcal L}}$ defined in \eqref{gp}.
\begin{thm}\label{cor:3}
Let $\rho$ be as in \eqref{rho}. Let $1<p<\infty$ and $0<\kappa<1$. If $V\in RH_q$ with $q\geq Q/2$, then the operator $\mathfrak{g}_{\sqrt{\mathcal L}}$ is bounded on $L^{p,\kappa}_{\rho,\infty}(\mathbb H^n)$.
\end{thm}

\begin{thm}\label{cor:4}
Let $\rho$ be as in \eqref{rho}. Let $p=1$ and $0<\kappa<1$. If $V\in RH_q$ with $q\geq Q/2$, then the operator $\mathfrak{g}_{\sqrt{\mathcal L}}$ is bounded from $L^{1,\kappa}_{\rho,\infty}(\mathbb H^n)$ to $WL^{1,\kappa}_{\rho,\infty}(\mathbb H^n)$.
\end{thm}

\section{Boundedness of the Lusin area integral $\mathcal{S}_{\mathcal L}$}\label{sec4}
In this section, we will study the boundedness properties of the Lusin area integral $\mathcal{S}_{\mathcal L}$ acting on $L^{p,\kappa}_{\rho,\infty}(\mathbb H^n)$ for $(p,\kappa)\in[1,\infty)\times(0,1)$. First recall the definition of the Lusin area integral $\mathcal{S}_{\mathcal L}$, which is given by
\begin{equation}\label{kauvw}
\begin{split}
\mathcal{S}_{\mathcal L}(f)(u)&=\bigg(\iint_{\Gamma(u)}\big|(s\mathcal L)e^{-s\mathcal L}f(v)\big|^2\frac{dvds}{s^{Q/2+1}}\bigg)^{1/2}\\
&=\bigg(\iint_{\Gamma(u)}\bigg|\int_{\mathbb H^n}\mathcal Q_s(v,w)f(w)\,dw\bigg|^2\frac{dvds}{s^{Q/2+1}}\bigg)^{1/2},
\end{split}
\end{equation}
where $\big\{e^{-s\mathcal L}\big\}_{s>0}$ is the semigroup generated by $\mathcal L$ and $\mathcal Q_{s}(v,w)$ denotes the kernel of $(s\mathcal L)e^{-s\mathcal L},s>0$.
Now we present the main results of this section.
\begin{thm}\label{mainthm:3}
Let $\rho$ be as in \eqref{rho}. Let $1<p<\infty$ and $0<\kappa<1$. If $V\in RH_q$ with $q\geq Q/2$, then the operator $\mathcal{S}_{\mathcal L}$ is bounded on $L^{p,\kappa}_{\rho,\infty}(\mathbb H^n)$.
\end{thm}

\begin{thm}\label{mainthm:4}
Let $\rho$ be as in \eqref{rho}. Let $p=1$ and $0<\kappa<1$. If $V\in RH_q$ with $q\geq Q/2$, then the operator $\mathcal{S}_{\mathcal L}$ is bounded from $L^{1,\kappa}_{\rho,\infty}(\mathbb H^n)$ to $WL^{1,\kappa}_{\rho,\infty}(\mathbb H^n)$.
\end{thm}

\begin{proof}[Proof of Theorem $\ref{mainthm:3}$]
For any given $f\in L^{p,\kappa}_{\rho,\infty}(\mathbb H^n)$ with $1\leq p<\infty$ and $0<\kappa<1$, suppose that $f\in L^{p,\kappa}_{\rho,\theta^*}(\mathbb H^n)$ for some $\theta^*>0$, where
\begin{equation*}
\theta^*=\inf\big\{\theta>0:f\in L^{p,\kappa}_{\rho,\theta}(\mathbb H^n)\big\}\quad\mathrm{and}\quad\big\|f\big\|_{L^{p,\kappa}_{\rho,\infty}(\mathbb H^n)}=\big\|f\big\|_{L^{p,\kappa}_{\rho,\theta^*}(\mathbb H^n)}.
\end{equation*}
By definition, we only need to show that for any given ball $B=B(u_0,r)$ of $\mathbb H^n$, there is some $\Theta>0$ such that
\begin{equation}\label{Main3}
\bigg(\frac{1}{|B(u_0,r)|^{\kappa}}\int_{B(u_0,r)}\big|\mathcal{S}_{\mathcal L}(f)(u)\big|^p\,du\bigg)^{1/p}\lesssim\left[1+\frac{r}{\rho(u_0)}\right]^{\Theta}
\end{equation}
holds true for given $f\in L^{p,\kappa}_{\rho,\theta^*}(\mathbb H^n)$ with $(p,\kappa)\in(1,\infty)\times(0,1)$. Using the standard technique, we decompose the function $f$ as
\begin{equation*}
\begin{cases}
f=f_1+f_2\in L^{p,\kappa}_{\rho,\theta^*}(\mathbb H^n);\  &\\
f_1=f\cdot\chi_{4B};\  &\\
f_2=f\cdot\chi_{(4B)^{\complement}},
\end{cases}
\end{equation*}
where $4B$ is the ball centered at $u_0$ of radius $4r>0$ and $(4B)^{\complement}=\mathbb H^n\backslash(4B)$. Then by the sublinearity of $\mathcal{S}_{\mathcal L}$, we write
\begin{equation*}
\begin{split}
\bigg(\frac{1}{|B(u_0,r)|^{\kappa}}\int_{B(u_0,r)}\big|\mathcal{S}_{\mathcal L}(f)(u)\big|^p\,du\bigg)^{1/p}
&\leq\bigg(\frac{1}{|B(u_0,r)|^{\kappa}}\int_{B(u_0,r)}\big|\mathcal{S}_{\mathcal L}(f_1)(u)\big|^p\,du\bigg)^{1/p}\\
&+\bigg(\frac{1}{|B(u_0,r)|^{\kappa}}\int_{B(u_0,r)}\big|\mathcal{S}_{\mathcal L}(f_2)(u)\big|^p\,du\bigg)^{1/p}\\
&:=I'_1+I'_2.
\end{split}
\end{equation*}
Let us estimate the first term $I'_1$. By Theorem \ref{strongs} (1) and \eqref{2rx}, we get
\begin{equation*}
\begin{split}
I'_1&\leq C\cdot\frac{1}{|B|^{\kappa/p}}\bigg(\int_{4B}\big|f(u)\big|^p\,du\bigg)^{1/p}\\
&\leq C\big\|f\big\|_{L^{p,\kappa}_{\rho,\theta^*}(\mathbb H^n)}\cdot
\frac{|4B|^{\kappa/p}}{|B|^{\kappa/p}}\left[1+\frac{4r}{\rho(u_0)}\right]^{\theta^*}\\
&\leq C_{\theta^*,n}\big\|f\big\|_{L^{p,\kappa}_{\rho,\infty}(\mathbb H^n)}\left[1+\frac{r}{\rho(u_0)}\right]^{\theta^*}.
\end{split}
\end{equation*}
We now estimate the second term $I'_2$. We first claim that the following inequality
\begin{equation}\label{keyWH2}
\mathcal{S}_{\mathcal L}(f_2)(u)\leq C_N\int_{(4B)^{\complement}}\bigg[1+\frac{|w^{-1}u|}{\rho(w)}\bigg]^{-N}\frac{1}{|w^{-1}u|^{Q}}\cdot|f(w)|\,dw
\end{equation}
holds for any $u\in B(u_0,r)$. Arguing as in the proof of Theorem \ref{mainthm:1}, two cases are considered below: $s>\frac{|w^{-1}u|^2}{4}$ and $0\leq s\leq\frac{|w^{-1}u|^2}{4}$. From \eqref{kauvw} and Lemma \ref{kernel}, it follows that
\begin{equation*}
\begin{split}
&\mathcal{S}_{\mathcal L}(f_2)(u)=\bigg(\iint_{\Gamma(u)}\bigg|\int_{\mathbb H^n}\mathcal Q_s(v,w)f_2(w)\,dw\bigg|^2\frac{dvds}{s^{Q/2+1}}\bigg)^{1/2}\\
&\leq\bigg(\int_0^{\infty}\int_{|u^{-1}v|<\sqrt{s}}\bigg|\int_{\mathbb H^n}\frac{C_N}{s^{Q/2}}\cdot\exp\bigg(-\frac{|w^{-1}v|^2}{As}\bigg)
\bigg[1+\frac{\sqrt{s}}{\rho(w)}\bigg]^{-N}|f_2(w)|\,dw\bigg|^2\frac{dvds}{s^{Q/2+1}}\bigg)^{1/2}\\
&=\bigg(\int_{\frac{|w^{-1}u|^2}{4}}^{\infty}\int_{|u^{-1}v|<\sqrt{s}}\bigg|\int_{\mathbb H^n}\frac{C_N}{s^{Q/2}}\cdot\exp\bigg(-\frac{|w^{-1}v|^2}{As}\bigg)
\bigg[1+\frac{\sqrt{s}}{\rho(w)}\bigg]^{-N}|f_2(w)|\,dw\bigg|^2\frac{dvds}{s^{Q/2+1}}\bigg)^{1/2}\\
&+\bigg(\int_0^{\frac{|w^{-1}u|^2}{4}}\int_{|u^{-1}v|<\sqrt{s}}\bigg|\int_{\mathbb H^n}\frac{C_N}{s^{Q/2}}\cdot\exp\bigg(-\frac{|w^{-1}v|^2}{As}\bigg)
\bigg[1+\frac{\sqrt{s}}{\rho(w)}\bigg]^{-N}|f_2(w)|\,dw\bigg|^2\frac{dvds}{s^{Q/2+1}}\bigg)^{1/2}\\
&:=\mathcal{S}_{\mathcal L}^{\infty}(f_2)(u)+\mathcal{S}_{\mathcal L}^0(f_2)(u).
\end{split}
\end{equation*}
When $s>\frac{|w^{-1}u|^2}{4}$, then $\sqrt{s\,}>\frac{|w^{-1}u|}{2}$, and hence
\begin{equation*}
\begin{split}
\mathcal{S}_{\mathcal L}^{\infty}(f_2)(u)&\leq\bigg(\int_{\frac{|w^{-1}u|^2}{4}}^{\infty}\int_{|u^{-1}v|<\sqrt{s}}\bigg|\int_{\mathbb H^n}\frac{C_N}{s^{Q/2}}\cdot
\bigg[1+\frac{\sqrt{s}}{\rho(w)}\bigg]^{-N}|f_2(w)|\,dw\bigg|^2\frac{dvds}{s^{Q/2+1}}\bigg)^{1/2}\\
&\leq\bigg(\int_{\frac{|w^{-1}u|^2}{4}}^{\infty}\int_{|u^{-1}v|<\sqrt{s}}\bigg|\int_{\mathbb H^n}\frac{C_N}{s^{Q/2}}\cdot
\bigg[1+\frac{|w^{-1}u|}{\rho(w)}\bigg]^{-N}|f_2(w)|\,dw\bigg|^2\frac{dvds}{s^{Q/2+1}}\bigg)^{1/2}.
\end{split}
\end{equation*}
By Minkowski's inequality, a straightforward computation yields that
\begin{equation*}
\begin{split}
\mathcal{S}_{\mathcal L}^\infty(f_2)(u)&\leq C_N\int_{\mathbb H^n}
\bigg[1+\frac{|w^{-1}u|}{\rho(w)}\bigg]^{-N}|f_2(w)|\bigg(\int_{\frac{|w^{-1}u|^2}{4}}^{\infty}\int_{|u^{-1}v|<\sqrt{s}}\frac{dvds}{s^{Q+Q/2+1}}\bigg)^{1/2}dw\\
&\leq C_{N,n}\int_{\mathbb H^n}\bigg[1+\frac{|w^{-1}u|}{\rho(w)}\bigg]^{-N}|f_2(w)|\bigg(\int_{\frac{|w^{-1}u|^2}{4}}^{\infty}\frac{ds}{s^{Q+1}}\bigg)^{1/2}dw\\
&\leq C_{N,n}\int_{(4B)^{\complement}}\bigg[1+\frac{|w^{-1}u|}{\rho(w)}\bigg]^{-N}\frac{1}{|w^{-1}u|^{Q}}\cdot|f(w)|\,dw.
\end{split}
\end{equation*}
On the other hand, it is easy to see that
\begin{equation*}
\begin{split}
&\int_{\mathbb H^n}\frac{C_N}{s^{Q/2}}\cdot\exp\bigg(-\frac{|w^{-1}v|^2}{As}\bigg)\bigg[1+\frac{\sqrt{s}}{\rho(w)}\bigg]^{-N}|f_2(w)|\,dw\\
&\leq\int_{\mathbb H^n}\frac{C_{N,A}}{s^{Q/2}}\cdot
\bigg(\frac{|w^{-1}v|^2}{s}\bigg)^{-(Q/2+N/2+\gamma/2)}\bigg[1+\frac{\sqrt{s}}{\rho(w)}\bigg]^{-N}|f_2(w)|\,dw\\
&=\int_{\mathbb H^n}\frac{C_{N,A}}{|w^{-1}v|^{Q}}\cdot
\bigg(\frac{s}{|w^{-1}v|^2}\bigg)^{\gamma/2}\bigg(\frac{\sqrt{s}}{|w^{-1}v|}\bigg)^{N}\bigg[1+\frac{\sqrt{s}}{\rho(w)}\bigg]^{-N}|f_2(w)|\,dw,
\end{split}
\end{equation*}
where $\gamma$ is a positive number. Note that when $0\leq s\leq\frac{|w^{-1}u|^2}{4}$ and $|u^{-1}v|<\sqrt{s}$, then $|w^{-1}v|\geq\frac{|w^{-1}u|}{2}$. Indeed, by triangle inequality,
\begin{equation}\label{wang1}
|w^{-1}v|\geq|w^{-1}u|-|u^{-1}v|\geq|w^{-1}u|-\frac{|w^{-1}u|}{2}=\frac{|w^{-1}u|}{2}.
\end{equation}
In addition, it is easy to check that when $0\leq s\leq\frac{|w^{-1}u|^2}{4}$,
\begin{equation}\label{wang2}
\frac{2\sqrt{s\,}}{|w^{-1}u|}\leq\frac{2\sqrt{s\,}+\rho(w)}{|w^{-1}u|+\rho(w)}.
\end{equation}
Hence, in view of \eqref{wang1} and \eqref{wang2}, we have
\begin{equation*}
\begin{split}
&\int_{\mathbb H^n}\frac{C_N}{s^{Q/2}}\cdot\exp\bigg(-\frac{|w^{-1}v|^2}{As}\bigg)\bigg[1+\frac{\sqrt{s}}{\rho(w)}\bigg]^{-N}|f_2(w)|\,dw\\
&\leq\int_{\mathbb H^n}\frac{C_N}{|w^{-1}u|^{Q}}\cdot
\bigg(\frac{s}{|w^{-1}u|^2}\bigg)^{\gamma/2}\bigg[\frac{\sqrt{s}+\rho(w)}{|w^{-1}u|+\rho(w)}\bigg]^{N}
\bigg[\frac{\sqrt{s}+\rho(w)}{\rho(w)}\bigg]^{-N}|f_2(w)|\,dw\\
&=\int_{\mathbb H^n}\frac{C_N}{|w^{-1}u|^{Q}}\cdot
\bigg(\frac{s}{|w^{-1}u|^2}\bigg)^{\gamma/2}\bigg[1+\frac{|w^{-1}u|}{\rho(w)}\bigg]^{-N}|f_2(w)|\,dw.
\end{split}
\end{equation*}
Therefore, applying this estimate along with Minkowski's inequality, we get
\begin{equation*}
\begin{split}
&\mathcal{S}_{\mathcal L}^0(f_2)(u)\\&\leq\bigg(\int_0^{\frac{|w^{-1}u|^2}{4}}\int_{|u^{-1}v|<\sqrt{s}}\bigg|\int_{\mathbb H^n}\frac{C_N}{|w^{-1}u|^{Q}}\cdot
\bigg(\frac{s}{|w^{-1}u|^2}\bigg)^{\gamma/2}\bigg[1+\frac{|w^{-1}u|}{\rho(w)}\bigg]^{-N}|f_2(w)|\,dw\bigg|^2\frac{dvds}{s^{Q/2+1}}\bigg)^{1/2}\\
&\leq\int_{\mathbb H^n}\frac{C_N}{|w^{-1}u|^{Q}}\cdot
\bigg[1+\frac{|w^{-1}u|}{\rho(w)}\bigg]^{-N}|f_2(w)|
\bigg(\int_0^{\frac{|w^{-1}u|^2}{4}}\int_{|u^{-1}v|<\sqrt{s}}\bigg(\frac{s}{|w^{-1}u|^2}\bigg)^{\gamma}\frac{dvds}{s^{Q/2+1}}\bigg)^{1/2}dw\\
\end{split}
\end{equation*}
\begin{equation*}
\begin{split}
&\leq\int_{\mathbb H^n}\frac{C_{N,n}}{|w^{-1}u|^{Q}}\cdot
\bigg[1+\frac{|w^{-1}u|}{\rho(w)}\bigg]^{-N}|f_2(w)|\bigg(\int_0^{\frac{|w^{-1}u|^2}{4}}\bigg(\frac{s}{|w^{-1}u|^2}\bigg)^{\gamma}\frac{ds}{s}\bigg)^{1/2}dw\\
&\leq C_{N,n,\gamma}\int_{(4B)^{\complement}}\bigg[1+\frac{|w^{-1}u|}{\rho(w)}\bigg]^{-N}\frac{1}{|w^{-1}u|^{Q}}\cdot|f(w)|\,dw.
\end{split}
\end{equation*}
Combining the above two estimates produces the desired inequality \eqref{keyWH2} for any $u\in B(u_0,r)$. Routine arguments as in the proof of Theorem \ref{mainthm:1} show that $|w^{-1}u|\approx|w^{-1}u_0|$ whenever $u\in  B(u_0,r)$ and $w\in (4B)^{\complement}$. This fact, along with \eqref{keyWH2} and \eqref{2rx}, implies that
\begin{equation*}
\begin{split}
\big|\mathcal{S}_{\mathcal L}(f_2)(u)\big|
&\leq C\int_{(4B)^{\complement}}\bigg[1+\frac{|w^{-1}u_0|}{\rho(w)}\bigg]^{-N}\frac{1}{|w^{-1}u_0|^{Q}}\cdot|f(w)|\,dw\\
&=C\sum_{k=2}^\infty\int_{2^kr\leq|w^{-1}u_0|<2^{k+1}r}\bigg[1+\frac{|w^{-1}u_0|}{\rho(w)}\bigg]^{-N}
\frac{1}{|w^{-1}u_0|^{Q}}\cdot|f(w)|\,dw\\
&\leq C\sum_{k=2}^\infty\frac{1}{|B(u_0,2^{k+1}r)|}
\int_{|w^{-1}u_0|<2^{k+1}r}\bigg[1+\frac{2^{k+1}r}{\rho(w)}\bigg]^{-N}|f(w)|\,dw.
\end{split}
\end{equation*}
Moreover, in view of \eqref{com2}, we obtain
\begin{align}\label{Talpha22}
\big|\mathcal{S}_{\mathcal L}(f_2)(u)\big|
&\leq C\sum_{k=2}^\infty\frac{1}{|B(u_0,2^{k+1}r)|}\notag\\
&\times\int_{B(u_0,2^{k+1}r)}\left[1+\frac{r}{\rho(u_0)}\right]^{N\cdot\frac{N_0}{N_0+1}}
\left[1+\frac{2^{k+1}r}{\rho(u_0)}\right]^{-N}|f(w)|\,dw.
\end{align}
We then follow the same arguments as in the proof of Theorem \ref{mainthm:1} to complete the proof.
\end{proof}

\begin{proof}[Proof of Theorem $\ref{mainthm:4}$]
By using the weak-type $(1,1)$ of $\mathcal{S}_{\mathcal L}$(Theorem \ref{strongs} (2)) and the pointwise estimate \eqref{Talpha22}, and following the proof of Theorem \ref{mainthm:2} line by line, we are able to prove the conclusion of Theorem \ref{mainthm:4}. The details are omitted here.
\end{proof}

Finally, we also consider the Lusin area integral with respect to the Poisson semigroup $\big\{e^{-s\sqrt{\mathcal L}}\big\}_{s>0}$, which is defined by
\begin{equation*}
\begin{split}
\mathcal{S}_{\sqrt{\mathcal L}}(f)(u)&=\bigg(\iint_{\Gamma(u)}\Big|s\cdot\frac{d}{ds}e^{-s\sqrt{\mathcal L}}f(v)\Big|^2\frac{dvds}{s^{Q/2+1}}\bigg)^{1/2}\\
&=\bigg(\iint_{\Gamma(u)}\big|(s\sqrt{\mathcal L})e^{-s\sqrt{\mathcal L}}f(v)\big|^2\frac{dvds}{s^{Q/2+1}}\bigg)^{1/2}.
\end{split}
\end{equation*}
As before, by the subordination formula \eqref{subor}, it is not difficult to check that this integral operator $\mathcal{S}_{\sqrt{\mathcal L}}$ is dominated by $\mathcal{S}_{\mathcal L}$ in some sense. As an immediate consequence of Theorems \ref{mainthm:3} and \ref{mainthm:4}, we have the following results.
\begin{thm}\label{cor:7}
Let $\rho$ be as in \eqref{rho}. Let $1<p<\infty$ and $0<\kappa<1$. If $V\in RH_q$ with $q\geq Q/2$, then the operator $\mathcal{S}_{\sqrt{\mathcal L}}$ is bounded on $L^{p,\kappa}_{\rho,\infty}(\mathbb H^n)$.
\end{thm}

\begin{thm}\label{cor:8}
Let $\rho$ be as in \eqref{rho}. Let $p=1$ and $0<\kappa<1$. If $V\in RH_q$ with $q\geq Q/2$, then the operator $\mathcal{S}_{\sqrt{\mathcal L}}$ is bounded from $L^{1,\kappa}_{\rho,\infty}(\mathbb H^n)$ to $WL^{1,\kappa}_{\rho,\infty}(\mathbb H^n)$.
\end{thm}

\section{Further remarks}\label{sec5}
In the last section, let us give the definitions of the generalized Morrey spaces. Let $\Phi=\Phi(r),r>0$, be a growth function; that is, a positive increasing function on $(0,\infty)$ and satisfy the following doubling condition
\begin{equation*}
\Phi(2r)\leq D\,\Phi(r),
\end{equation*}
for all $r>0$, where $D=D(\Phi)\geq1$ is a doubling constant independent of $r$.
\begin{defin}
Let $\rho$ be the auxiliary function determined by $V\in RH_q$ with $q\geq Q/2$. Let $1\leq p<\infty$ and $\Phi$ be a growth function. For given $0<\theta<\infty$, the generalized Morrey space $L^{p,\Phi}_{\rho,\theta}(\mathbb H^n)$ is defined to be the set of all $p$-locally integrable functions $f$ on $\mathbb H^n$ such that
\begin{equation}\label{morrey3}
\bigg(\frac{1}{\Phi(|B(u_0,r)|)}\int_{B(u_0,r)}|f(u)|^p\,du\bigg)^{1/p}
\leq C\cdot\left[1+\frac{r}{\rho(u_0)}\right]^{\theta}
\end{equation}
holds for every ball $B(u_0,r)$ in $\mathbb H^n$, and we denote the smallest constant $C$ satisfying \eqref{morrey3} by $\|f\|_{L^{p,\Phi}_{\rho,\theta}(\mathbb H^n)}$.It is easy to see that the functional $\|\cdot\|_{L^{p,\Phi}_{\rho,\theta}(\mathbb H^n)}$ is a norm on the linear space $L^{p,\Phi}_{\rho,\theta}(\mathbb H^n)$ that makes it into a Banach space under this norm. Define
\begin{equation*}
L^{p,\Phi}_{\rho,\infty}(\mathbb H^n):=\bigcup_{0<\theta<\infty}L^{p,\Phi}_{\rho,\theta}(\mathbb H^n).
\end{equation*}
\end{defin}

\begin{defin}
Let $\rho$ be the auxiliary function determined by $V\in RH_q$ with $q\geq Q/2$. Let $p=1$ and $\Phi$ be a growth function. For given $0<\theta<\infty$, the generalized weak Morrey space $WL^{1,\Phi}_{\rho,\theta}(\mathbb H^n)$ is defined to be the set of all measurable functions $f$ on $\mathbb H^n$ such that
\begin{equation}\label{morrey4}
\frac{1}{\Phi(|B(u_0,r)|)}\sup_{\lambda>0}\lambda\cdot\big|\big\{u\in B(u_0,r):|f(u)|>\lambda\big\}\big|
\leq C\cdot\left[1+\frac{r}{\rho(u_0)}\right]^{\theta}
\end{equation}
holds for every ball $B(u_0,r)$ in $\mathbb H^n$, and we denote the smallest constant $C$ satisfying \eqref{morrey4} by $\|f\|_{WL^{1,\Phi}_{\rho,\theta}(\mathbb H^n)}$.
Correspondingly, we define
\begin{equation*}
WL^{1,\Phi}_{\rho,\infty}(\mathbb H^n):=\bigcup_{0<\theta<\infty}WL^{1,\Phi}_{\rho,\theta}(\mathbb H^n).
\end{equation*}
\end{defin}

\begin{rem}
$(i)$ As in Section \ref{def}, we can also define a norm and a quasi-norm on the linear spaces $L^{p,\Phi}_{\rho,\infty}(\mathbb H^n)$ and $WL^{1,\Phi}_{\rho,\infty}(\mathbb H^n)$, respectively.

$(ii)$ According to this definition, we recover the spaces $L^{p,\kappa}_{\rho,\theta}(\mathbb H^n)$ and $WL^{1,\kappa}_{\rho,\theta}(\mathbb H^n)$ under the choice $\Phi(t)=t^\kappa$, for all $(p,\kappa)\in[1,\infty)\times(0,1)$.

$(iii)$ In the Euclidean setting, when $\theta=0$ or $V\equiv0$, the classes $L^{p,\Phi}_{\rho,\theta}$ and $WL^{1,\Phi}_{\rho,\theta}$ reduce to the classes $L^{p,\Phi}$ and $WL^{1,\Phi}$, which were introduced and studied by Mizuhara in \cite{mizuhara}.
\end{rem}

Using the similar method as in the proofs of Theorems \ref{mainthm:1} through \ref{mainthm:4}, we can also prove the following results. The details are omitted here.
\begin{thm}\label{mainthm:9}
Let $\rho$ be as in \eqref{rho}. Let $1<p<\infty$ and $1\leq D(\Phi)<2$. If $V\in RH_q$ with $q\geq Q/2$, then the operators $\mathfrak{g}_{\mathcal L}$ and $\mathcal{S}_{\mathcal L}$ are all bounded on $L^{p,\Phi}_{\rho,\infty}(\mathbb H^n)$.
\end{thm}

\begin{thm}\label{mainthm:10}
Let $\rho$ be as in \eqref{rho}. Let $p=1$ and $1\leq D(\Phi)<2$. If $V\in RH_q$ with $q\geq Q/2$, then the operators $\mathfrak{g}_{\mathcal L}$ and $\mathcal{S}_{\mathcal L}$ are all bounded from $L^{1,\Phi}_{\rho,\infty}(\mathbb H^n)$ to $WL^{1,\Phi}_{\rho,\infty}(\mathbb H^n)$.
\end{thm}
Note that the same conclusions of Theorems \ref{mainthm:9} and \ref{mainthm:10} also hold for the operators $\mathfrak{g}_{\sqrt{\mathcal L}}$ and $\mathcal{S}_{\sqrt{\mathcal L}}$ with respect to the Poisson semigroup $\big\{e^{-s\sqrt{\mathcal L}}\big\}_{s>0}$.

\bibliographystyle{elsarticle-harv}
\bibliography{<your bibdatabase>}

\end{document}